\documentclass[12pt,a4paper]{article}

\usepackage{amsmath,amssymb,amsthm}
\usepackage{mathtools}
\usepackage[shortlabels]{enumitem}
\setlist[enumerate]{label=(\roman*)}
\usepackage{graphicx}
\usepackage{tikz}
\usepackage{tikz-cd}
\usepackage[colorlinks=false]{hyperref}
\hypersetup{
    colorlinks,
    linkcolor={red!50!black},
    citecolor={blue!50!black},
    urlcolor={blue!80!black}
}
\usepackage[capitalise]{cleveref}
\usepackage{xcolor}
\usepackage{natbib}
\setcitestyle{semicolon}

\newtheorem{theorem}{Theorem}
\newtheorem{lemma}[theorem]{Lemma}
\newtheorem{prop}[theorem]{Proposition}
\newtheorem{cor}[theorem]{Corollary}
\theoremstyle{definition}
\newtheorem{defn}[theorem]{Definition}
\newtheorem{example}[theorem]{Example}
\newtheorem*{remark}{Remark}

\newcommand \dm[1]  { \,\mathrm d{#1} }

\renewcommand{\epsilon}{\varepsilon}
\renewcommand{\phi}{\varphi}
\renewcommand{\emptyset}{\varnothing}
\renewcommand{\leq}{\leqslant}
\renewcommand{\geq}{\geqslant}

\newcommand{\norm}[1]{\left\lVert#1\right\rVert}

\newcommand{\norml}[3]{\norm{#1}_{L^{#2}({#3})}}

\newcommand{\normw}[3]{\norm{#1}_{W^{#2}({#3})}}

\newcommand{\pair}[1]{\left\langle{#1}\right\rangle}

\newcommand{\abs}[1]{\left\vert#1\right\vert}

\renewcommand{\hat}[1]{\widehat{#1}}

\renewcommand{\tilde}{\widetilde}
\newcommand{\e}{\mathrm{e}}
\renewcommand{\i}{\mathrm{i}}
\renewcommand{\Re}{\mathrm{Re}}

\newcommand{\tpi}{2 \pi \mathrm{i}}

\newcommand{\R}{\mathbb{R}}

\newcommand{\C}{\mathbb{C}}
\newcommand{\N}{\mathbb{N}}

\newcommand{\Z}{\mathbb{Z}}

\newcommand{\Cc}{C_c^\infty}

\newcommand{\qtext}[1]{\quad\text{#1}\quad}
\newcommand{\qand}{\qtext{and}}
\newcommand{\qfa}{\qtext{for all}}
\newcommand{\ie}{i.e.\ }
\newcommand{\eg}{e.g.}

\title{$L^p$ continuity of eigenprojections for 2-d Dirichlet Laplacians under perturbations of the domain}
\author{Ryan L. Acosta Babb\thanks{r.acosta-babb@warwick.ac.uk, University of Warwick, UK (corresponding author)}%
\and James C. Robinson\thanks{j.c.robinson@warwick.ac.uk, University of Warwick, UK}}
\date{}

\begin{document}
\maketitle
\tableofcontents

\begin{abstract}
    We generalise results by \cite{Lamberti2005}
    concerning the continuity of projections onto eigenspaces of self-adjoint differential operators
    with compact inverses as the (spatial) domain of the functions is perturbed in $\R^2$.
    Our main case of interest is the Dirichlet Laplacian on a square.
    We extend their results from bounds from $H_0^1$ to $H_0^1$
    to bounds from $L^p$ to $L^p$, under the assumption
    that $(-\Delta^{-1}-z)^{-1}$ is $L^p$ bounded when $z$ lies outside of the spectrum of $-\Delta^{-1}$.
    We show that this assumption is met if the initial domain is a square or a rectangle.
\end{abstract}

\section{Introduction}
It is a classical consequence of the Uniform Boundeness Principle
that the norm convergence of Fourier series is equivalent to
the uniform boundedness of the operator norms of the partial sums
\citep[see][Chap.\,4]{GrafakosCFA}.
In dimensions 2 or higher, Fefferman's celebrated theorem \citep{Fefferman1971} on the unboundedness of the ball multiplier in $L^p$
has as a consequence that we cannot choose the partial sums in any way we want,
since convergence is sensitive to the choice of ``truncation'' for these infinite series.

For concreteness, let us fix our dimension $d=2$ and consider the lattice $\Z^2$,
which serves to index the Fourier coefficients of a periodic function $f:[0,1]^2\to\R$.
Then there are many ways in which we can group these indices in order to determine
the cutoff point for a partial sum of the series
\[
    \sum_{(m,n)\in\Z^2}\hat{f}(m,n)\e^{\tpi(mx+ny)}.
\]
A natural choice would be to truncate by the size of the eigenvalues associated to each
exponential (viewed as an eigenfunction of the Laplacian).
Up to constants, this would mean fixing $N>0$ and summing over all $(m,n)$ with $m^2+n^2<N^2$.
Geometrically, we are using expanding concentric circles in the lattice to determine the cutoff.
As we mentioned earlier, Fefferman's theorem shows that this method never yields $L^p$ convergence
outside the ``trivial'' case $p=2$.
It is a straightforward extension of this result that other ``curved'' cutoffs,
such as ellipses, will also fail to yield convergence
\citep[see][Chap.\,5 and its exercises for a discussion]{Grafakos2014MFA}.

On the other hand, taking squares or rectangles to mark the cutoff does work for all $1<p<\infty$,
as is well-known \citep[see][Chap.\, 4]{GrafakosCFA}.
Results by \cite{Cordoba1977} show that other polygonal regions also generate ``good'' cutoffs
within the range $4/3<p<4$.

So far we have been interested in the geometry of ``frequency space'', \ie the lattice $\Z^2$.
All results above concern functions defined on the torus, \ie periodic functions on the square.
The scaling properties of Fourier series imply that all results carry over to functions
defined on rectangles $[0,a]\times[0,b]$.
Recently, one of the authors has transferred $L^p$ convergence results
to functions defined on a certain class of triangular domains \citep{Acosta2023TFs}.
Another case of interest is functions defined on the disc,
which, to the best of our knowledge, remains open
\citep[see][for a discussion of some of the issues]{Acosta2023BF}.

One way in which we could explore the issue of convergence in other planar domains
is by perturbation.
The idea is to take a planar region $\Omega$ and deform it into another region $\tilde{\Omega}$.
We then study how the eigenfunctions of the Laplacian
and the projection operators onto eigenspaces
change alongside this change to $\Omega$.
The hope is that, for sufficiently small perturbations
we can carry over convergence results on a domain where they are known
(\eg\,a square, a triangle)
to new domains where the problem is still unsolved.
The importance of having such $L^p$ convergence results for problems in PDE
is discussed in \cite{RobinsonFefferman2022}.
They note that a general understanding of $L^p$ convergence for eigenfunction expansions
is still open, prompting our current efforts.

We now discuss the setup of the perturbation approach in some more detail,
following \cite{Lamberti2005}.

Let $\Omega\subset\R^2$ be a bounded, open and connected domain in the plane
(so, in particular, the Poincar\'e inequality holds on $\Omega$).
We perturb $\Omega$ by a map $\phi\colon\Omega\to\tilde{\Omega}$
close to the identity in a sense to be made precise later.
We are interested in relating the Dirichlet eigenproblem on $\tilde{\Omega}$ back to that on $\Omega$.
We wish to compare the two operators by applying them to functions defined on $\Omega$.
Thus, ``$\Delta$'' will denote the Laplacian acting on $u\in H_0^1(\Omega)$,
while ``$\Delta_\phi$'' denotes the Laplacian acting on $u\circ\phi^{-1}\in H_0^1(\tilde{\Omega})$,
and then ``pulled back'' via $\phi$ to $H_0^1(\Omega)$.
Thus, $\Delta - \Delta_\phi$ is well-defined as an operator acting on functions defined on $\Omega$.

We list the (positive) eigenvalues of $-\Delta$, \emph{including multiplicities}, as \[
    0 < \lambda_1 \leq \lambda_2 \leq \lambda_3 \leq \ldots.
\] Similarly, we denote the eigenvalues of $-\Delta_\phi$ by ``$\tilde{\lambda}_j$''.

Note that, since $0$ is not an eigenvalue,
the Laplacian and its inverse yield equivalent eigenvalue problems:
\begin{equation}\label{eqn:equiveval}
    -\Delta u =\lambda u \qtext{if and only if} -\Delta^{-1}u = \lambda^{-1}u.
\end{equation}
We will therefore concentrate on the latter problem, since the inverse operators
are compact and self-adjoint \citep[see][Lemma 3.5]{Lamberti2005}.

We will need to impose a further technical condition on the eigenspaces we consider.
Fix a finite subset $F\subset \N$ and denote by ``$P_F$'' the orthogonal projection in $L^2$
onto the eigenfunctions $u_j$ associated to the $\lambda_j$ with $j\in F$.
Note, in particular, that $u_j\neq u_{j+1}$ with $\lambda_{j}=\lambda_{j+1}$
if $\lambda_j$ has a multiplicity of at least 2.

\begin{defn}
    We say that $F$ \emph{splits}
    an eigenvalue $\lambda$ if there are indices $j,k\in\N$ such that $\lambda = \lambda_j = \lambda_k$
    and $j\in F$ but $k\notin F$.
\end{defn}

In terms of the projection $P_F$, this means that $P_Fu_j = u_j$ but $P_Fu_{k}=0$
even though $-\Delta u_{j} = \lambda u_j$ and $-\Delta u_{k} = \lambda u_{k}$.
We need to rule out eigenvalue splitting in order to apply Kato's projection formula
(see the remark following \cref{thm:hilbertspaceproj} for further details).

In \cref{sec:prelims} we set up the necessary notation
and discuss the perturbation of the domain in some detail.
In \cref{sec:evals} we prove that under certain assumptions on $\phi$,
$F$ will not split eigenvalues of $\Delta_\phi$ if it did not split eigenvalues of $\Delta$.
\cref{sec:operators} contains the main estimates.
In \cref{sec:main} we state and prove our main result for the Dirichlet Laplacian:
$L^p$ to $L^p$ boundedness of the projections under perturbations of the domain $\Omega$,
assuming boundedness of the resolvents of $\Delta^{-1}$.
The brief \cref{sec:resolvent} is devoted to showing that this assumption is met
for square and rectangular domains.

Throughout we will use the notation $A(s) \lesssim B(s)$ to mean that there is a constant $C>0$,
which does not depend on $s$, such that \[
    A(s) \leq B(s) \qfa s.
\]
\section{Preliminaries}\label{sec:prelims}

We begin with a homeomorphism $\phi\colon \Omega\to \tilde{\Omega}$ and we wish to study
operators on both of these domains. Since these operators act on different
function spaces, \eg\,$H_0^1(\Omega)$ and $H_0^1(\tilde{\Omega})$,
we cannot compare them directly.

Let $u,v\in H_0^1(\Omega)$. Their $L^2$ inner product is
\begin{equation}\label{eqn:L2Om}
    \pair{u,v}_{L^2(\Omega)} = \int_{\Omega}uv\dm{x}.
\end{equation}
On the other hand, $u\circ\phi^{-1},v\circ\phi^{-1}\in H_0^1(\tilde{\Omega})$,
with $L^2$ inner product
\begin{align}\label{eqn:L2PhiOm}
    \pair{u\circ\phi^{-1},v\circ\phi^{-1}}_{L^2(\tilde{\Omega})}
        &= \int_{\phi(\Omega)}u(\phi^{-1}(y))v(\phi^{-1}(y))\dm{y}\nonumber\\
        &= \int_{\Omega}u(x)v(x)\abs{\det{D\phi}(x)}\dm{x} =: Q_\phi[u,v].
\end{align}
If we denote the standard inner product on $L^2(\Omega)$ by ``$Q$'',
then we can compare the inner products from \cref{eqn:L2Om,eqn:L2PhiOm}:
\begin{equation}\label{eqn:Qdiff}
    (Q-Q_\phi)[u,v] = \int_{\Omega}uv(1-\abs{\det{D\phi}})\dm{x} \qfa u,v\in H_0^1(\Omega).
\end{equation}

We begin by recording a key result in \cite{Lamberti2005} which we are going to generalise in this paper.
Let $H$ be a Hilbert space equipped with two different bilinear forms, $Q$ and $\tilde{Q}$.
Suppose that $T$ and $\tilde{T}$ are two compact operators from $H$ to itself
and self-adjoint on $(H,Q)$ and $(H,\tilde{Q})$, respectively.
Let $F\subset\N$ be a finite set of indices and denote by $P_F[Q,T]$ the projection
\[
    P_F[Q,T]u := \sum_{j\in F}Q[u,u_j]u_j
\] where $u_j$ are eigenfunctions of $T$ corresponding to the indices $j\in F$.

\begin{theorem}[\cite{Lamberti2005}]\label{thm:hilbertspaceproj}
Let $F\subset \N$ be a finite set of indices that does not split eigenvalues of $T$ or $\tilde{T}$.
Then, there is a positive constant $C = C(T,\tilde{T},F)$ such that
\begin{equation}
    \norm{P_F[Q,T]-P_F[\tilde{Q},\tilde{T}]}_{H\to H}
    \leq C\left(\norm{Q-\tilde{Q}}_{H\times H\to\R}+\norm{T-\tilde{T}}_{H\to H}\right).
\end{equation}
\end{theorem}
\begin{remark}
    The main tool in the proof is Kato's integral formula for a projection:
    \[
        P_F[Q,T] = \Re\left(-\frac{1}{\tpi}\int_{\gamma[F,T]}(T-z I)^{-1}\dm{z}\right).
    \] Here $\gamma[F,T]$ is a contour encircling the portion of the spectrum of $T$
    which contains the eigenvalues in $F$.

    As we can see, the integral formula is ``blind'' to multiplicities:
    it only differentiates the values of the eigenvalues, and there is no way
    of keeping track of exactly which eigenfunctions we want to keep in our projection.
    This is the reason behind the no-splitting condition, since if our set $F$ were to split
    an eigenvalue, then Kato's formula would not actually give us the projection $P_F[T,Q]$,
    but rather a projection onto a larger space which includes all eigenfunctions for every eigenvalue ``in $F$''.
\end{remark}

We will use a similar argument in which $T = \Delta^{-1}$ and $\tilde{T}$
is a pullback to $\Omega$ of the inverse Laplacian on $\tilde{\Omega}$.
Thus, there are three Laplacians in play here:
the Laplacian $\Delta = \Delta_\Omega$ acting on $\Cc(\Omega)$;
the Laplacian $\Delta_{\tilde{\Omega}}$ acting on $\Cc(\tilde{\Omega})$;
and a transformed Laplacian $\Delta_\phi$ acting also on $\Cc(\Omega)$ as follows:
\[
    \Delta_\phi f(x) := (\Delta_{\tilde{\Omega}} (f\circ \phi^{-1}))(\phi(x)) \qfa f\in \Cc(\Omega).
\]
That is: the function $f\circ\phi^{-1}$ is defined on $\tilde{\Omega}$,
which we then differentiate using $\Delta_{\tilde{\Omega}}$,
and finally pull back to $\Omega$ by composing the result with $\phi$.

We will denote the eigenfunctions of $\Delta_\Omega$ and $\Delta_{\tilde{\Omega}}$
respectively by $u_j$ and $\tilde{u}_j$.
The eigenfunctions of $\Delta_\phi$ are $u^\phi_j := \tilde{u}_j\circ\phi$
\citep[see][Theorem 3.10]{Lamberti2005}.

We are now in a position to distinguish three ``partial sum'' operators of interest:
\begin{enumerate}
    \item Partial sums on $\Omega$, which are expansions in $L^2(\Omega)$
    in terms of the eigenfunctions $u_j$ of $\Delta_{\Omega}$ relative to $\Omega$:
    \begin{equation}\label{eqn:SF}
        P_Ff = \sum_{j\in F}\pair{u_j,f}_{\Omega}u_j \qfa f \in L^2(\Omega).
    \end{equation}
    \item Partial sums on $\tilde{\Omega}$, which are expansions in $L^2(\tilde{\Omega})$
    in terms of the eigenfunctions $\tilde{u}_j$ of $\Delta_{\tilde{\Omega}}$ relative to $\tilde{\Omega}$:
    \begin{equation}\label{eqn:SFt}
        \tilde{P}_F g = \sum_{j\in F}\pair{\tilde{u}_j, g}_{\tilde{\Omega}}\tilde{u}_j \qfa g \in L^2(\tilde{\Omega}).
    \end{equation}
    \item Transformed partial sums on $\Omega$ which are ``pull-backs'' of $\tilde{S}_F$ to $L^2(\Omega)$.
    These are expansions with respect to the eigenfunctions $u^\phi_j$ of $\Delta_\phi$ in $L^2(\Omega)$,
    now equipped with the inner product $Q_\phi$:
    \begin{equation}\label{eqn:SFphi}
        P^\phi_Ff = \sum_{j\in F}Q_\phi[u^\phi_j, f]u^\phi_j \qfa f \in L^2(\Omega).
    \end{equation}
\end{enumerate}

Most of the work will go into obtaining bounds $\norml{T-\tilde{T}}{p}{\Omega}$
which depend (i) on the change of variables map $\phi\colon \Omega\to \tilde{\Omega}$,
and (ii) on the set $F$.
In order to do so, we first need to make sense of $\Delta^{-1}$ as a bounded operator on $L^p$.

\subsection{The inverse of the Laplacian}\label{sec:fnspaces}
Under certain assumptions on the regularity of $\Omega$,
the Laplacian \[\Delta\colon W^{2,p}(\Omega)\cap W^{1,p}_0(\Omega)\to L^p(\Omega)\]
is invertible; see, for example, Theorem 9.15 in \cite{Gilbarg2001}.
\begin{theorem}\label{thm:PoissonC11}
    Let $\Omega$ be a $C^{1,1}$ domain in $\R^2$.
    Then, if $f\in L^p(\Omega)$ with $1<p<\infty$, the Dirichlet problem $\Delta u = f$ in $\Omega$
    with $u\in W_0^{1,p}(\Omega)$ has a unique solution $u\in W^{2,p}(\Omega)$.
\end{theorem}

A similar result also holds for a square, as follows from Theorem 4.4.3.7 in \cite{Grisvard85}
(whose statement we slightly adapt).
\begin{theorem}\label{thm:Poissonpolys}
    Assume that $\Omega$ is a polygon with straight boundary edges meeting at right angles.
    Then, for each $f\in L^p(\Omega)$ there is a unique solution $u \in W^{2,p}(\Omega)\cap W_0^{1,p}(\Omega)$
    to the Poisson problem $\Delta u = f$ with zero boundary conditions.
\end{theorem}

These are the more ``obvious'' examples; for a discussion on more general conditions on $\partial\Omega$
for the $W^{2,p}$-solvability of the Poisson equation with $L^p$ initial data,
see Chapter 7 of \cite{Mazya85}.

In other words: for squares, rectangles and $C^{1,1}$ domains (among others),
the Dirichlet Laplacian is invertible with inverse $\Delta^{-1}\colon L^p\to W^{2,p}\cap W_0^{1,p}$.

For brevity, we will henceforth write $X_p$ for the Banach space $W^{2,p}\cap W_0^{1,p}$
equipped with the $W^{2,p}$ norm.

\begin{defn}
    Let $\phi\colon \Omega\to \tilde{\Omega}$ be an invertible map
    with bounded weak derivatives up to second order.
    Furthermore, we assume that
    \[
        \inf_\Omega \abs{\det{D\phi}} > 0.
    \]
    We call the triple $(\Omega,\phi,\tilde{\Omega})$ \emph{$p$-admissible}
    if $\phi$ is as above and $\Delta\colon X_p\to L^p$ is invertible
    on both $\Omega$ and $\tilde{\Omega}$.
    We will often simplify this to ``$\phi\colon\Omega\to\tilde{\Omega}$ is $p$-admissible''
    or even ``$\phi$ is $p$-admissible''.
\end{defn}
\begin{remark}
    We require one more order of differentiation than \cite{Lamberti2005},
    as well as the invertibility of $\Delta$ from $X_p$ to $L^p$,
    which comes ``for free'' on bounded domains in $\R^2$
    when working with $\Delta\colon H^1_0\to H^{-1}$.
\end{remark}

Henceforth, we will assume that $(\Omega,\phi,\tilde{\Omega})$ is $p$-admissible.
\begin{defn}
    Let $Y(\cdot)$ be a function space on $\Omega$ and $\tilde{\Omega}$, and let $\phi\colon \Omega\to \tilde{\Omega}$
    be a change of coordinates.
    We define the following operators:

    \noindent\begin{minipage}{.4\linewidth}
        \begin{align*}
            C_{\phi}\colon Y(\tilde{\Omega}) &\longrightarrow Y(\Omega)\\
            v \quad&\longmapsto v\circ\phi
        \end{align*}
    \end{minipage}%
    \begin{minipage}{.2\linewidth}
        \[ \text{and} \]
    \end{minipage}%
    \begin{minipage}{.4\linewidth}
        \begin{align*}
            C_{\phi^{-1}}\colon Y(\Omega) &\longrightarrow Y(\tilde{\Omega})\\
            u \quad&\longmapsto u\circ\phi.
        \end{align*}
    \end{minipage}
\end{defn}
\begin{remark}
    Note that $(C_{\phi})^{-1} = C_{\phi^{-1}}$ and vice versa.
\end{remark}

We will denote by $T$ the inverse Laplacian
\[ T = \Delta^{-1}\equiv (\Delta_{\Omega})^{-1}\colon L^p(\Omega)\to X_p(\Omega)\hookrightarrow L^p(\Omega).\]
We will also denote by $T_\phi$ the inverse of $\Delta_\phi$, \ie the map \[
    T_\phi = C_\phi \circ \Delta_{\tilde{\Omega}}^{-1} \circ C_{\phi^{-1}} \colon L^p(\Omega)\to L^p(\Omega).\]
Thus, we have the following diagram: \[
\begin{tikzcd}
    L^p(\tilde{\Omega}) \arrow[r, "\Delta_{\tilde{\Omega}}^{-1}"]            & X_p(\tilde{\Omega}) \arrow[r, hook] & L^p(\tilde{\Omega}) \arrow[d, "C_\phi"] \\
    L^p(\Omega) \arrow[r, "\Delta_{\Omega}^{-1}"] \arrow[u, "C_{\phi^{-1}}"] & X_p(\Omega) \arrow[r, hook]         & L^p(\Omega)
    \end{tikzcd}
\] where $T_\phi$ is the ``up--over--down'' path from $L^p(\Omega)$ to $L^p(\Omega)$.

We now introduce a slightly modified version of the operator $\Delta_\phi$.
\begin{defn}\label{defn:Lphi}
    Let $\mathcal{L}_\phi\colon L^p(\Omega)\to L^p(\Omega)$ be the operator defined by \[
        \mathcal{L}_\phi := C_{\phi^{-1}}^{t} \circ \Delta_{\tilde{\Omega}} \circ C_{\phi^{-1}},
    \] where $C_{\phi^{-1}}^t\colon L^p(\tilde{\Omega})\to L^p(\Omega)$ is the operator defined, for $v\in L^p(\tilde{\Omega})$, by \[
        \pair{C_{\phi^{-1}}^tv, w} := \pair{v, C_{\phi^{-1}}w}
            \equiv \int_{\tilde{\Omega}}v(y)w(\phi^{-1}(y))\dm{y} \qfa w \in L^q(\Omega).
    \]
\end{defn}
\begin{lemma}
    The operator $C_{\phi^{-1}}^t$ is given by \[
        \left[C_{\phi^{-1}}^{t}v\right](x) = \abs{\det{D\phi(x)}}[C_\phi v](x) \qfa v \in L^p(\tilde{\Omega}),
    \] with inverse \[
        \left[(C_{\phi^{-1}}^t)^{-1}u\right](y) = C_\phi^{-1}\left(\frac{u(x)}{\abs{\det{D\phi(x)}}}\right)(y) \qfa u \in L^p(\Omega).
    \] In particular, \begin{equation}\label{eqn:Cphitinv}
        (C_{\phi^{-1}}^{t})^{-1} = C_\phi^{t}.
    \end{equation}
\end{lemma}
\begin{defn}
    Define the operator $\mathcal{J}_\phi\colon L^p(\Omega) \to L^p(\Omega)$ by \[
        \mathcal{J}_\phi u(x) := \abs{\det{D\phi(x)}}u(x) \qfa u \in L^p(\Omega).
    \] If $\inf_\Omega\abs{\det{D\phi}}>0$, then \[
        \mathcal{J}_\phi^{-1}u(x) = \frac{u(x)}{\abs{\det{D\phi(x)}}}.
    \]
\end{defn}
We can therefore conclude the following.
\begin{cor}\label{cor:TphiCt}
    We have $T_\phi = \mathcal{L}_{\phi}^{-1} \circ \mathcal{J}_\phi$.
\end{cor}
\begin{proof}
    First, note that $\mathcal{L}_\phi = C_{\phi^{-1}}^{t}\circ\Delta_{\tilde{\Omega}}\circ C_{\phi^{-1}}$
    implies that
    \begin{align*}
        (\mathcal{L}_{\phi})^{-1} &= (C_{\phi^{-1}})^{-1} \circ (\Delta_{\tilde{\Omega}})^{-1} \circ (C_{\phi^{-1}}^t)^{-1}\\
            &= C_\phi \circ \Delta_{\tilde{\Omega}}^{-1} \circ C_\phi^{-1}\circ \mathcal{J}_\phi^{-1}\\
            &= (\Delta_\phi)^{-1}\circ \mathcal{J}_\phi^{-1},
    \end{align*}
    and so composing with $\mathcal{J}_\phi$ finishes the proof.
\end{proof}
These rather formal calculations are justified by the admissibility conditions on $\phi$.

\begin{remark}
    The crucial detail to observe is that we factor the (transformed) inverse Laplacian
    $T_\phi$ into two operators:
    the inverse ``Laplacian'' $\mathcal{L}_\phi^{-1}\colon L^p(\Omega)\to L^p(\Omega)$
    and the embedding $\mathcal{J}_\phi\colon L^p(\Omega)\to L^p(\Omega)$.
    This circuitous route, which is inspired by \cite{Lamberti2005},
    will vastly simplify our proof of \cref{prop:DeltaSdiff} below.
    Essentially, we will need to transfer integration by parts
    from $\Omega$ to $\tilde{\Omega}$.
    If we were simply to use $\Delta_\phi$, we would have
    \begin{align*}
        \pair{\Delta_\phi u, w} &= \pair{C_\phi\Delta_{\tilde{\Omega}} C_{\phi^{-1}}u, w}\\
        &= \int_{\Omega}(\Delta_{\tilde{\Omega}}(u\circ\phi^{-1}))(\phi(x))w(x)\dm{x}\\
        &= \int_{\tilde{\Omega}}\Delta_{\tilde{\Omega}}(u\circ\phi^{-1})(y)w(\phi^{-1}(y))\abs{\det D\phi^{-1}(y)}\dm{y}\\
        &= -\int_{\tilde{\Omega}}D_y(u\circ\phi^{-1})(y) \cdot D_y\left[(w\circ \phi^{-1})\abs{\det{D\phi^{-1}(y)}}\right]\dm{y}.
    \end{align*}
    Using the operator $\mathcal{L}_\phi$ instead will allow us to circumvent differentiating the determinant,
    which is then reintroduced by $\mathcal{J}_\phi$ in a more convenient place.
\end{remark}
\section{Continuity and splitting of eigenvalues}\label{sec:evals}
In order to apply Kato's projectio formula, we must assume that the set $F$
does not split eigenvalues of either operator.
In this section, we show that, under reasonable assumptions on $\phi$,
it is enough to assume that $F$ does not split eigenvalues of $\Delta$ on $\Omega$.
For in this case, continuity results of the eigenvalues with respect to changes of the domain
will ensure that $F$ will not split the eigenvalues of $\Delta_{\tilde{\Omega}}$ on $\tilde{\Omega}$.
The next lemma guarantees that $\Delta_{\tilde{\Omega}}$ and its pull-back
to $\Omega$ have the same eigenvalues.
Thus, if $F$ does not split eigenvalues of $T$, then
it will not split eigenvalues of $T_\phi$,
and we can apply Kato's projection formula.
(Recall that $\lambda>0$ is an eigenvalue of $-\Delta$ if, and only if,
$\frac{1}{\lambda}>0$ is an eigenvalue of $-\Delta^{-1}$.)

\begin{lemma}\label{lemma:Tphieigenvals}
    Let $\lambda\in\R$ and $v\in L^2(\tilde{\Omega})\setminus\{0\}$.
    Then, \[
        \Delta_{\tilde{\Omega}}^{-1} v = \lambda v \qtext{if, and only if,} T_\phi(C_\phi v) = \lambda C_\phi v.
    \]
\end{lemma}
\begin{proof}
    Immediate from the definitions; see Theorem 3.10 in \cite{Lamberti2005}.
\end{proof}

\cref{lemma:Tphieigenvals} tells us that the eigenvalues of $\Delta_{\tilde{\Omega}}^{-1}$
remain unchanged after we ``pull them back to $\Omega$''---that is the purpose of $T_\phi$.

We next check that the eigenvalues of $\Delta^{-1}$
are close to the eigenvalues of $\Delta_{\tilde{\Omega}}^{-1}$.
This we do via a theorem due to \cite{Courant1989} (see p.\,423).
\begin{theorem}\label{thm:CourantHilbert}
    Let $\phi:\Omega\to\tilde{\Omega}$ be a bijection with $\inf_{x\in\Omega}\abs{\det{D\phi}}>0$.
    Then, for any $\eta>0$ there is an $\epsilon>0$ such that if \[
        \normw{\operatorname{id}-\phi}{1,\infty}{\Omega} \leq \epsilon,
    \] then \[
        \abs{\frac{\tilde{\lambda}_j}{\lambda_j} - 1} < \eta
        \qfa j\in\N.
    \]
\end{theorem}
\begin{remark}
    Strictly speaking, this result applies to $-\Delta$, not $-\Delta^{-1}$.
    But we have seen that the eigenvalue problems for both operators are equivalent
    (see \cref{eqn:equiveval} and the surrounding discussion in the Introduction).
    The proof of this result uses the variational formulation of the eigenvalue problem.
    As \cite{Lamberti2005} point out, the ``variational eigenvalues'' need not coincide
    with the true eigenvalues.
    However, this is the case in bounded domains.
    (More generally, when the embedding $H_0^1(\Omega)\hookrightarrow L^2(\Omega)$ is compact;
    see their discussion on p.\,289.)
\end{remark}
\begin{lemma}\label{lemma:noslice}
    Let $j,k\in F$ be such that $\lambda_j\neq\lambda_k$.
    Assume that \[\gamma:=\inf\{\lambda_l-\lambda_m:l>m>0\}>0.\]
    Then, $\abs{\tilde{\lambda}_j - \tilde{\lambda}_k} > \gamma/2$,
    provided that $\normw{\operatorname{id}-\phi}{1,\infty}{\Omega}$ is small enough.
\end{lemma}
\begin{proof}
Without loss of generality, we may assume that $\lambda_j > \lambda_k$,
and set $\delta := \lambda_j-\lambda_k\geq\gamma$.
Note that \[
    \tilde{\lambda}_j - \tilde{\lambda}_k
    = \lambda_j\cdot \frac{\tilde{\lambda}_j}{\lambda_j}
        - \lambda_k\cdot\frac{\tilde{\lambda}_k}{\lambda_k}
    = (\lambda_j-\lambda_k)\cdot\frac{\tilde{\lambda}_j}{\lambda_j}
        + \lambda_k\left(\frac{\tilde{\lambda}_j}{\lambda_j}
            -\frac{\tilde{\lambda}_k}{\lambda_k}\right).
\] By \cref{thm:CourantHilbert}, \[
    1-\eta < \frac{\tilde{\lambda}_j}{\lambda_j} < 1+\eta,
    \quad \abs{\frac{\tilde{\lambda}_j}{\lambda_j}
    -\frac{\tilde{\lambda}_k}{\lambda_k}} < 2\eta \qand \abs{\lambda_k} \leq \max_{l\in F}\abs{\lambda_l}.
\] Hence,
\[
    \tilde{\lambda}_j - \tilde{\lambda}_k
        \geq (\lambda_j-\lambda_k)(1-\eta) - 2\max_{l\in F}\abs{\lambda_l} \eta
        \geq \gamma -\eta(\delta+2\max_{l\in F}\abs{\lambda_l}),
\] which is larger than $\gamma/2$, provided that $\eta<\frac{\gamma}{\delta+2\max_{l\in F}\abs{\lambda_l}}$.
\end{proof}

Since there are only finitely many indices in $F$, it follows that, for a fixed $F$
and a sufficiently small perturbation $\phi$,
the multiplicities of the eigenvalues of $\Delta_{\tilde{\Omega}}$ (in $F$)
will not decrease with respect to the multiplicities of the eigenvalues of $\Delta$ (in $F$).

Combining \cref{lemma:noslice} with \cref{lemma:Tphieigenvals},
we deduce that $F$ does not split eigenvalues for $T$ nor $T_\phi$.
\begin{cor}\label{cor:noslice}
    If $\phi$ is $p$-admissible and
    $\normw{\operatorname{id}-\phi}{1,\infty}{\Omega}$ sufficiently small,
    then $F$ will not split eigenvalues of $T$ nor $T_\phi$.
\end{cor}
\section{Bounds on the norms of the operators}\label{sec:operators}

\subsection{Auxiliary bounds}
We need bounds on the operators $C_\phi\colon L^p(\tilde{\Omega})\to L^p(\Omega)$
and $C_{\phi}^t\colon L^p(\Omega)\to L^p(\tilde{\Omega})$.

\begin{lemma}\label{lem:Cphibounds}
    If $\phi$ is $p$-admissible, then
    \begin{equation*}\label{eqn:Cphiboud}
        \norm{C_\phi}_{L^p(\tilde{\Omega})\to L^p(\Omega)}
            \leq \frac{1}{\inf_{\Omega}\abs{\det{D\phi}}^{1/p}},
    \end{equation*}
    and, with $1/p+1/q=1$,
    \begin{equation}\label{eqn:Cphitbound}
        \norm{C_{\phi}^t}_{L^p(\Omega)\to L^p(\tilde{\Omega})}
            \leq \frac{1}{\inf_\Omega\abs{\det{D\phi}}^{1/q}}.
    \end{equation}
\end{lemma}
\begin{proof}
    For any $v\in L^p(\tilde{\Omega})$ we have
    \begin{align*}
        \norml{C_\phi v}{p}{\Omega}^p &=
        \int_{\Omega}\abs{v(\phi(x))}^p\dm{x}\\
        &= \int_{\Omega}\abs{v(\phi(x))}^p
            \frac{\abs{\det{D\phi}}}{\abs{\det{D\phi}}}\dm{x}\\
        &\leq \frac{1}{\inf_\Omega\abs{\det{D\phi}}}
        \int_{\Omega}\abs{v(\phi(x))}^p\abs{\det{D\phi}}\dm{x}\\
        &= \frac{1}{\inf_\Omega\abs{\det{D\phi}}}
        \int_{\tilde{\Omega}}\abs{v(y)}^p\dm{y},
    \end{align*}
    by the change of variables formula, and
    \eqref{eqn:Cphitbound} immediately follows by a simple duality argument.
\end{proof}

\begin{cor}\label{cor:DeltaSphibounds}
    If $\phi$ is $p$-admissible, then
    \[
        \norm{\mathcal{L}_{\phi}^{-1}}_{L^p(\Omega)\to L^p(\Omega)}
            \lesssim \frac{1}{\inf_\Omega{\abs{\det{D\phi}}}}.
    \]
\end{cor}
\begin{proof}
    Recall that, by \cref{eqn:Cphitinv}, we have \[
        \mathcal{L}_{\phi}^{-1} = C_\phi \circ \Delta_{\tilde{\Omega}}^{-1}\circ C_\phi^t.
    \] Hence, the result follows from \cref{lem:Cphibounds} and the estimate
    \[\norm{\Delta_{\tilde{\Omega}}^{-1}}_{L^p(\tilde{\Omega})\to L^p(\tilde{\Omega})}\lesssim 1,\]
    by the $p$-admissibility of $\tilde{\Omega}$.
\end{proof}

Next we need some technical lemmas concerning $\phi$ and its derivatives.
We will write $\phi = \mathrm{id} + (f,g)$ to mean \[
    \phi(x,y) = (x + f(x,y), y + g(x,y)).
\]
We want to consider $\phi$ ``close to the identity'' in the sense that
$\normw{\mathrm{id}-\phi}{2,\infty}{\Omega;\R^2}$ is small.
In other words, $f,g$ and their derivatives up to second order have small $L^\infty$ norms.
\begin{lemma}\label{lemma:detDphi}
    Suppose that $\normw{\mathrm{id}-\phi_k}{1,\infty}{\Omega;\R^2}\to 0$ as $k\to\infty$.
    Then,
    \[\norml{1-\det{D\phi_k}}{\infty}{\Omega}\to 0.\]
\end{lemma}
\begin{proof}
    Let \[
        D\phi_k = \left(\begin{matrix}
            a_k & b_k\\
            c_k & d_k
        \end{matrix}\right).
    \] Then, by assumption, \[
        D\phi_k-I = \left(\begin{matrix}
            a_k-1 & b_k\\
            c_k & d_k-1
        \end{matrix}\right) \to 0 \qtext{in} L^\infty.
    \] Hence, $\det{(D\phi_k-I)} \to 0$
    and it follows that $1-\det{(D\phi_k-I)} \to 0$.
\end{proof}
\begin{remark}
    Note that we only required $\mathrm{id}-\phi\in W^{1,\infty}(\Omega;\R^2)$ for this lemma.
    The next result will require convergence in $W^{2,\infty}$.
\end{remark}
\begin{lemma}\label{lemma:Kphi}
    Suppose that $\normw{\mathrm{id}-\phi_k}{2,\infty}{\Omega;\R^2}\to 0$ as $k\to\infty$
    and write \[
        M_k := D\phi_k^{-1} D\phi_k^{-t}\abs{\det{D\phi}} - I \equiv
        \left(\begin{matrix}
            a_k & b_k\\
            c_k & d_k
        \end{matrix}\right).
    \] Then, as $k\to\infty$,
    \begin{enumerate}
        \item $\norml{M_k}{\infty}{\Omega} \to 0$, and
        \item $\norml{\frac{\partial a_k}{\partial x}+\frac{\partial b_k}{\partial y}}{\infty}{\Omega}$,
        $\norml{\frac{\partial c_k}{\partial x}+\frac{\partial d_k}{\partial y}}{\infty}{\Omega} \to 0$.
    \end{enumerate}
\end{lemma}
\begin{proof}
    Recall that $\phi_k = \mathrm{id} + (f_k,g_k)$.
    For simplicity, we will drop the $k$ subscripts on $\phi$, $f$ and $g$,
    writing $f_x := \frac{\partial f_k}{\partial x}$,
    and so on.
    Thus, \eg, ``$f_x\to 0$'' will mean ``$\frac{\partial f_k}{\partial x}\to 0$ as $k\to\infty$''.
    Note that our assumption means precisely that $f_x,f_y,g_x,g_y\to 0$,
    as well as all second partial derivatives.

    By \cref{lemma:detDphi}, eventually $\det{D\phi}>0$ a.e.,
    so we may drop the absolute value signs.
    A simple calculation shows that \[
        D\phi^{-1} D\phi^{-t} = \frac{1}{(\det{D\phi})^2}\left(\begin{matrix}
            g_x^2 + (1+g_y)^2 & -(f_y+g_x+f_xg_y+f_yg_x)\\
            -(f_y+g_x+f_xg_y+f_yg_x) & (1+f_x)^2+f_y^2
        \end{matrix}\right)
    \] and that \[
        \det{D\phi} = 1 + f_x + g_y + f_xg_y + f_yg_x.
    \] Hence, for example, the top-left entry of $M_k$ is\[
        a_k = \frac{g_x^2+g_y^2-g_y-f_x-f_xg_y-f_yg_x}{\det{D\phi}}.
    \] By \cref{lemma:detDphi}, the denominator remains bounded,
    while the numerator vanishes as $k\to\infty$ by assumption.
    Similar calculations for $b,c$ and $d$ establish (i).

    For (ii), we differentiate the last display with respect to $x$ to find \[
        \frac{\partial a_k}{\partial x} = \frac{p(Df,Dg,D^2f,D^2g)}{(\det{D\phi})^2},
    \] where $p$ is a homogeneous polynomial in the first and second derivatives of $f$ and $g$.
    Thus, as above, $\partial a_k/\partial x\to 0$. We omit the remaining cases.
\end{proof}

\subsection{Bounds on $T-T_\phi$}
With these estimates, we can now control $T-T_\phi$.
\begin{prop}\label{prop:DeltaSdiff}
    If $\phi\colon \Omega\to \tilde{\Omega}$ is $p$-admissible, then
    \begin{align*}
        \norm{T-T_\phi}_{L^p(\Omega)\to L^p(\Omega)}
        \lesssim \frac{1}{\inf_\Omega \abs{\det{D\phi}}}\left(
                \norml{1-\abs{\det{D\phi}}}{\infty}{\Omega}+K_\phi\right),
    \end{align*}
    where $K_\phi$ is a constant depending on $D\phi$ and its derivatives
    which vanishes as $\normw{\mathrm{id}-\phi}{2,\infty}{\Omega}\to 0$.
\end{prop}
\begin{proof}
    Fix $u\in L^p(\Omega)$.
    Recall that $T_\phi = \mathcal{L}_\phi^{-1}\circ\mathcal{J}_\phi$.
    Hence, it follows from \cref{cor:DeltaSphibounds} that
    \begin{align}
        \norml{Tu-T_\phi u}{p}{\Omega}
        &\leq \norm{\mathcal{L}_\phi^{-1}}_{L^p(\Omega)\to L^p(\Omega)}
            \norml{\mathcal{J}_\phi u-\mathcal{L}_\phi\Delta^{-1}u}{p}{\Omega}\notag\\
        &\lesssim \frac{1}{\inf_{\Omega}\abs{\det{D\phi}}}
        \norml{\mathcal{J}_\phi u-\mathcal{L}_\phi\Delta^{-1}u}{p}{\Omega}\label{eqn:TTphi}.
    \end{align}
    To bound the last $L^p$ norm, we use duality.
    For convenience, let $Au := \mathcal{J}_\phi u-\mathcal{L}_\phi \Delta^{-1}u$,
    and compute, for $w\in \Cc(\Omega)$ with $\norml{w}{q}{\Omega}\leq 1$:
    \begin{align}\label{eqn:Auw}
        \pair{Au,w}
        &= \pair{\mathcal{J}_\phi u,w} -\pair{\mathcal{L}_\phi\Delta^{-1}u,w}\notag\\
        \begin{split}
            &= \int_{\Omega}uw\abs{\det{D\phi}}\dm{x}
            -\pair{\mathcal{L}_\phi\Delta^{-1}u,w}.
        \end{split}
    \end{align}
    We now unravel the definition of $\mathcal{L}_\phi$, integrate by parts and change variables:
    \begin{align*}
        \pair{\mathcal{L}_\phi[\Delta^{-1}u],w}
            &= \pair{C_{\phi^{-1}}^t\Delta_{\tilde{\Omega}}C_\phi^{-1}[\Delta^{-1} u],w}\\
            &= \pair{\Delta_{\tilde{\Omega}}C_{\phi^{-1}}[\Delta^{-1} u], C_{\phi^{-1}}w}\\
            &= -\int_{\tilde{\Omega}}D_y\left[\Delta^{-1} u(\phi^{-1}(y))\right]\cdot
                D_y\left[w(\phi^{-1}(y))\right]\dm{y}\\
            &= -\int_{\tilde{\Omega}}D_x (\Delta^{-1} u)(\phi^{-1}(y))D\phi^{-1}(y)D\phi^{-t}(y)(D_xw)^{t}(\phi^{-1}(y))\dm{y}\\
            &= -\int_{\Omega}D(\Delta^{-1}u)D\phi^{-1}(D\phi^{-1})^{t}(Dw)^{t}\abs{\det{D\phi}}\dm{x},
    \end{align*}
    where for a matrix $M$, $M^{-t} = (M^{-1})^t$, and $Dv \cdot Dw = Dv(Dw)^t$.
    We will now add and subtract
    \[
        \pm\int_{\Omega}D(\Delta^{-1}u)\cdot Dw\dm{x}
    \] from \eqref{eqn:Auw} to obtain:
    \begin{align}
        \pair{Au,w} = &\int_{\Omega}uw\abs{\det{D\phi}}\dm{x} + \int_{\Omega}D(\Delta^{-1}u)\cdot Dw\dm{x}\label{eqn:I}\\
        \begin{split}
            &\int_{\Omega}D(\Delta^{-1}u)D\phi^{-1}D\phi^{-t}(Dw)^{t}\abs{\det{D\phi}}\dm{x}\\
            &\quad -\int_{\Omega}D(\Delta^{-1}u)\cdot Dw\dm{x}.\label{eqn:II}
        \end{split}
    \end{align}

    To estimate \eqref{eqn:I}, we integrate by parts in the second summand to see that
    \begin{align}
        \eqref{eqn:I} &= \int_{\Omega}uw\abs{\det{D\phi}}\dm{x}
            -\int_{\Omega}\Delta(\Delta^{-1}u)w\dm{x}\notag\\
            &= \int_{\Omega}uw(\abs{\det{D\phi}}-1)\dm{x}\notag\\
            &\leq \norml{1-\abs{\det{D\phi}}}{\infty}{\Omega}\norml{u}{p}{\Omega}\norml{w}{q}{\Omega}.\label{eqn:Iestimate}
    \end{align}

    For the second term, note that
    \[
        \eqref{eqn:II} = \int_\Omega D(\Delta^{-1}u)\left\{D\phi^{-1}(D\phi^{-1})^t\abs{\det{D\phi}}-I\right\}\cdot Dw\dm{x},
    \] where $I$ denotes the $2\times 2$ identity matrix.
    Writing the expression in braces as \[
        M_\phi = \left(\begin{matrix}
            a & b\\
            c & d
        \end{matrix}\right)
    \] and integrating by parts yields
    \[
        \eqref{eqn:II} =
            -\int_{\Omega}\left[(a_x+b_y, c_x+d_y)\cdot D(\Delta^{-1}u) +
            M_\phi\cdot D^2(\Delta^{-1}u)\right]w\dm{x},
    \] where $A\cdot B$ is the dot product of the matrices $A$ and $B$,
    seen as vectors in $\R^4$.
    Then,
    \begin{align}
        \eqref{eqn:II} &\lesssim K_\phi\left(\norml{D^2(\Delta^{-1}u)}{p}{\Omega}
            +\norml{D(\Delta^{-1}u)}{p}{\Omega}\right)\norml{w}{q}{\Omega}\notag\\
        &\leq K_\phi\normw{\Delta^{-1}u}{2,p}{\Omega}\norml{w}{q}{\Omega}\notag\\
        &\lesssim K_\phi\norml{u}{p}{\Omega}\norml{w}{q}{\Omega}\label{eqn:IIestimate}.
    \end{align}
    The last line follows from the boundedness of $\Delta^{-1}\colon L^p(\Omega)\to X_p(\Omega)$,
    while $K_\phi$ is a constant depending on $D\phi$ and its derivatives,
    which vanishes in the required manner by \cref{lemma:Kphi}.

    Hence, combining the estimates \eqref{eqn:Iestimate} and \eqref{eqn:IIestimate}
    with \eqref{eqn:Auw} and \eqref{eqn:TTphi} yields the result.
\end{proof}
\begin{remark}
    If $\phi\colon \Omega\to\tilde{\Omega}$ is conformal,
    then we can regard it as a holomorphic map from $\C$ to $\C$
    whose complex derivative $\phi'(x+\i y)$ does not vanish at any $(x,y)\in \Omega$.
    In such a case we have \[
        \det{D\phi} = \abs{\phi'}^2 \qand
        \frac{1}{\abs{\det{D\phi}}}D\phi D\phi^t = I,
    \] whence it follows that \[
        D\phi^{-1}D\phi^{-t}\abs{\det{D\phi}} = I,
    \] and so \cref{eqn:Auw} simplifies to
    \begin{align*}
        \pair{Au,w} &= \int_{\Omega}uw\abs{\det{D\phi}}\dm{x}+\int_{\Omega}D[\Delta^{-1}u]\cdot Dw\dm{x}\\
                    &= \int_{\Omega}uw\left(\abs{\det{D\phi}}-1\right)\dm{x}.
    \end{align*}
\end{remark}
\begin{cor}\label{cor:phihol}
    Suppose that $\phi\colon\Omega\to\tilde{\Omega}$ is $p$-admissible, conformal and differentiable.
    Then, \[
        \norm{T-T_\phi}_{L^p(\Omega)\to L^p(\Omega)}
        \lesssim \frac{1}{\inf_\Omega \abs{\det{D\phi}}}\norml{1-\abs{\det{D\phi}}}{\infty}{\Omega}.
    \]
\end{cor}

\section{Continuity of projections under domain perturbations}\label{sec:main}
We will now combine the estimates for $T-T_\phi$
with Kato's integral formula to obtain $L^p$ a continuity theorem
for the projections as $\mathrm{id}-\phi\to 0$ in $W^{2,\infty}$,
provided one can establish $L^p$ bounds on the resolvents on $T=\Delta^{-1}$.
In the next section we will show that this is possible when $\Omega$ is a square of rectangular domain.

Before we state the theorem, we note that
the projection operators for $\Delta$ and its inverse $T$
are identical, since they share the same eigenfunctions;
and similarly for $\Delta_\phi$ and $T_\phi$.

\begin{theorem}\label{thm:deltaproj2}
    Suppose that $\phi\colon \Omega\to\tilde{\Omega}$ is $p$-admissible,
    and $F$ does not split eigenvalues of $T$ or $T_\phi$.
    Suppose further that
    \begin{equation}\label{eqn:resolventhyp}
        \norm{(T-z)^{-1}}_{L^p\to L^p} \leq C(T,z)
    \end{equation}
    for all $z\in\C\setminus\sigma(\tilde{T})$.
    Then,
    \begin{equation*}
        \norm{P_{F}-P_{F}^\phi}_{L^p(\Omega)\to L^p(\Omega)}
        \lesssim
        C(F,\Omega,\phi)\left(\frac{\norml{1-\abs{\det{D\phi}}}{\infty}{\Omega}+K_\phi}{\inf_{\Omega}\abs{\det{D\phi}}}\right)
    \end{equation*}
    where $K_\phi$ is the constant from \cref{lemma:Kphi}
    and the implicit constants depend only on $p$.

    Furthermore, if $\phi_k$ is a sequence of $p$-admissible transformations,
    such that $\normw{\mathrm{id}-\phi_k}{2,\infty}{\Omega}\to 0$,
    then $\norm{P_{F}-P_{F}^{\phi_k}}_{L^p(\Omega)\to L^p(\Omega)}\to 0$
    for each fixed $F$ that does not split eigenvalues of $\Delta$.
\end{theorem}
\begin{proof}
    Under the assumptions on $\Omega$ and $\phi$,
    $T$ and $T_\phi$ are well-defined compact operators on $L^2(\Omega)$.
    Label their eigenvalues as $(\mu_j)_{j\in\N}$ and $(\tilde{\mu}_j)_j$,
    respectively, and recall that $\mu_j=1/\lambda_j$, etc.

    For any $u \in \Cc(\Omega)$, we have the $L^2(\Omega)$ expansions \[
        T u = \sum_{j\in\N}\mu_j\pair{u_j,u}_{\Omega}u_j \qand
        T_\phi = \sum_{j\in\N}\tilde{\mu}_j Q_\phi[u^{\phi}_j,u] u^{\phi}_j.
    \] For any $z\in\C$ distinct from any $\mu_j$ and $\tilde{\mu}_j$, we have (still in $L^2$) \[
        (T-z)^{-1} u = \sum_{j\in\N}\frac{\pair{u_j,u}_{\Omega}}{\mu_j-z}u_j \qand
        (T_\phi-z)^{-1} = \sum_{j\in\N}\frac{Q_\phi[u^{\phi}_j,u]}{\tilde{\mu}_j-z}u^{\phi}_j.
    \]

    Now select a contour $\gamma$ in the complex plane enclosing precisely the
    $\mu_j$ and $\tilde{\mu_j}$ with $j\in F$ and no others.
    This is possible by \cref{thm:CourantHilbert}.

    By Kato's projection formula,
    \begin{align}\label{eqn:projdiff}
        P_Fu - P_F^{\phi}u &= \frac{1}{\tpi}\int_{\gamma}\left\{(T-z)^{-1}-(T_\phi-z)^{-1}\right\}u\dm{z}\nonumber\\
        &= \frac{1}{\tpi}\int_\gamma(T_\phi-z)^{-1}(T_\phi-T)(T-z)^{-1}u\dm{z}
    \end{align}

    We want to replace the first resolvent, $(T_\phi-z)^{-1}$, with a truncated version:
    \[
        R^\phi_F(z)f:= \sum_{j\in F}\frac{Q_\phi[u^{\phi}_j,u]}{\tilde{\mu}_j-z}u_j^\phi.
    \] Indeed, we claim that:
    \begin{equation}\label{eqn:CIFclaim}
        \frac{1}{\tpi}\int_{\gamma}(T_\phi-z)^{-1}f\dm{z}
            = \frac{1}{\tpi}\int_{\gamma}R_F^\phi(z)f\dm{z} \qfa f \in L^2(\Omega).
    \end{equation}
    It suffices to check that \cite[VII,\S4]{Conway2007} \[
        \frac{1}{\tpi}\int_{\gamma}Q_\phi[(T_\phi-z)^{-1}f,g]\dm{z}
        = \frac{1}{\tpi}\int_{\gamma}Q_\phi[R_F^\phi(z)f,g]\dm{z}
    \] for all $f$ and $g$ in $L^2(\Omega)$.

    For $f\in L^2(\Omega)$, we can write $f = \sum_j Q_\phi[u^\phi_j, f]u^\phi_j$
    and apply the bounded operator $(T_\phi-z)^{-1}$:
    \begin{align*}
        \frac{1}{\tpi}\int_\gamma Q_\phi[(T_\phi-z)^{-1}f,g]\dm{z}
        &= \frac{1}{\tpi}\int_\gamma\sum_j\frac{Q_\phi[(T_\phi-z)^{-1}f,g]}{\tilde{\mu}_j-z}Q_\phi[u_j^\phi,g]\dm{z}\\
        &= \sum_jQ_\phi[u_j^\phi,f]Q_\phi[u_j^\phi,g]\left(\frac{1}{\tpi}\int_{\gamma}\frac{1}{\tilde{\mu}_j-z}\dm{z}\right).
    \end{align*}
    By the classical Cauchy integral formula, the integral in brackets is
    $1$ if $j\in F$ and $0$ otherwise. Hence, the last expression
    is equal to \[
        \frac{1}{\tpi}\int_{\gamma}Q_\phi\left[\sum_{j\in F}\frac{Q_\phi[u_j^\phi,u_j]}{\tilde{\mu}_j-z}u_j^\phi,g\right]\dm{z},
    \] and the claim follows.

    The next step is to obtain $L^p$ bounds on the operator norm of $R_F^\phi(z)$.
    Taking $f\in \Cc$, which dense in both $L^2$ and $L^p$, we have \[
        \norml{R_F^\phi(z)f}{p}{\Omega}
        \leq \sum_{j\in F}\frac{\abs{Q_\phi[u_j^\phi,f]}}{\abs{\tilde{\mu}_j-z}}\norml{u_j^\phi}{p}{\Omega}.
    \] The change of variables formula immediately shows that (for $1/p+1/q=1$)
    \[
        \norml{u_j^\phi}{p}{\Omega} \leq \frac{\norml{\tilde{u}_j}{p}{\tilde{\Omega}}}{\inf_\Omega{\abs{\det{D\phi}}^{1/p}}}
    \] and \[
        \abs{Q_\phi[u_j^\phi,f]} \leq \frac{\norml{\det{D\phi}}{\infty}{\Omega}\norml{\tilde{u}_j}{q}{\tilde{\Omega}}}{\inf_\Omega{\abs{\det{D\phi}}^{1/q}}}\norml{f}{p}{\Omega}.
    \] Hence, it follows that \[
        \norm{R_F^\phi(z)}_{L^p(\Omega)\to L^p(\Omega)} \leq
            \frac{1}{\min_F\abs{\tilde{\mu}_j-z}}
            \frac{\norml{\det{D\phi}}{\infty}{\Omega}}{\inf_\Omega{\abs{\det{D\phi}}}}
            \sum_{j\in F}\norml{\tilde{u}_j}{p}{\tilde{\Omega}}\norml{\tilde{u}_j}{q}{\tilde{\Omega}}.
    \]

    Let us write
    \begin{equation}\label{eqn:Rzconst}
        C(F,\phi) := \frac{\norml{\det{D\phi}}{\infty}{\Omega}}{\inf_\Omega{\abs{\det{D\phi}}}}
        \sum_{j\in F}\norml{\tilde{u}_j}{p}{\tilde{\Omega}}\norml{\tilde{u}_j}{q}{\tilde{\Omega}}.
    \end{equation}
    Then, applying (\ref{eqn:CIFclaim}) to (\ref{eqn:projdiff}), and recalling assumption (\ref{eqn:resolventhyp})
    yields:
    \begin{equation*}
        \norml{P_Fu-P_F^\phi u}{p}{\Omega}\lesssim C(F,\phi)\int_{\gamma}\frac{C(T,z)}{\min_F{\abs{\tilde{\mu}_j-z}}}\dm{z}
            \norm{T-T_\phi}_{L^p(\Omega)\to L^p(\Omega)}\norml{u}{p}{\Omega}.
    \end{equation*}

    We can package the first two terms into a constant $C(F,T,\phi)$.
    On the other hand, \cref{prop:DeltaSdiff} gave us the estimate
    \begin{align*}
        \norm{T-T_\phi}_{L^p(\Omega)\to L^p(\Omega)}
        \lesssim \frac{1}{\inf_\Omega \abs{\det{D\phi}}}\left(
                \norml{1-\abs{\det{D\phi}}}{\infty}{\Omega}+K_\phi\right),
    \end{align*}
    and the first result immediately follows.

    The continuity statement is an immediate consequence of \cref{lemma:detDphi}
    and \cref{lemma:Kphi}.
\end{proof}

\cref{cor:phihol} gives us the following simplification when $\phi$ is conformal.
\begin{cor}
    Suppose, in addition, that $\phi$ is holomorphic. Then,
        \[
            \norm{P_{F}-P_{F}^\phi}_{L^p(\Omega)\to L^p(\Omega)}
            \lesssim
            C(F)\norml{1-\abs{\det{D\phi}}}{\infty}{\Omega}\times E_\phi,
        \]
        where $E_\phi$ remains bounded as $\mathrm{id}-\phi\to 0$ in $W^{2,\infty}$.
\end{cor}
\begin{remark}
    When $p=2$, the term \[
        \sum_{j\in F}\norm{u_j}_{L^p}\norm{u_j}_{L^q}
    \] appearing in \cref{eqn:Rzconst} can be improved to $1$ \citep[see][eq.\,(13)]{Lamberti2005}.
    Indeed, for any compact operator $T$, \[
        \norm{(T-z)^{-1}f}^2_{L^2} = \sum_j\frac{\abs{\pair{u_j,f}}^2}{\abs{\mu_j-z}^2}
        \leq \frac{1}{\min{\abs{\mu_j-z}}^2}\norm{f}^2_{L^2}.
    \]
    In their paper, \cite{Lamberti2005} claim to obtain this bound from
    the following inequality \citep[VI.3, Theorem 3.1]{TaylorLay86}: \[
        \norm{(T-z)^{-1}} \leq \frac{1}{d(z, V(T))},
    \] where $d(z,V(T))$ is the distance of $z\in\gamma$
    to the \emph{numerical range} of the operator $T$, \[
        V(T) := \left\{(Tx,x) : \norm{x} = 1\right\}.
    \] However, in a Hilbert Space \citep[p.\,116]{Halmos82}, $V(T)$
    is the closed convex hull of the spectrum of $T$.
    Since $T$ is compact, $0$ lies in the closure of the spectrum,
    and so the contour $\gamma$ as described will necessarily intersect
    the numerical range when it crosses the real axis.
    Our alternative calculation has the advantage of generalising to
    $L^p$ for the truncated operators $R_F^\phi(z)$,
    at the cost of increasing the size of the constants.
\end{remark}

Under suitable assumptions on $\phi$, one can use $P^\phi_F$ to transfer bounds on $P_F\colon L^p(\Omega)\to L^p(\Omega)$
to bounds on $\tilde{P}_F\colon L^p(\tilde{\Omega})\to L^p(\tilde{\Omega})$.
This ``transference'' result is encapsulated in the following lemma.
\begin{lemma}[Transference from $\Omega$ to $\tilde{\Omega}$]
    Let $\phi\colon\Omega\to\tilde{\Omega}$ be (weakly) differentiable, invertible and satisfy the bounds \[
        0 < \inf_\Omega \abs{\det{D\phi}} \leq \sup_\Omega\abs{\det{D\phi}} < \infty.
    \]
    Suppose that we have bounds of the form
    \[
        \norm{P_F}_{L^p(\Omega)\to L^p(\Omega)} \leq A
    \] and \[
        \norm{P_F - P_F^\phi}_{L^p(\Omega)\to L^p(\Omega)} \leq B.
    \] Then, we have the bound \[
        \norm{\tilde{P}_F}_{L^p(\tilde{\Omega})\to L^p(\tilde{\Omega})}
            \leq \left(\frac{\sup\abs{\det{D\phi}}}{\inf\abs{\det{D\phi}}}\right)^{1/p}(A+B).
    \]
\end{lemma}
\begin{proof}
    Fix $f\in L^p(\Omega)$ and let $g := f\circ \phi^{-1}\in L^p(\tilde{\Omega})$.
    Then, by a change of variables in the integral defining the inner product on $\tilde{\Omega}$, we have
    \[
        \tilde{P}_Fg = \sum_{j\in F}
            \left(\int_{\Omega}\tilde{u}_j(\phi(x))f(x)\abs{\det{D\phi(x)}}\dm{x}\right)\tilde{u}_j
            = \sum_{j\in F} Q_\phi[\tilde{u}_j\circ\phi, f] \tilde{u}_j.
    \] Taking $L^p(\tilde{\Omega})$ norms and changing variables again:
    \begin{align*}
        \norml{\tilde{P}_F f}{p}{\tilde{\Omega}}^p &=
            \int_{\tilde{\Omega}}\abs{\sum_{j\in F}Q_\phi[\tilde{u}_j\circ\phi, f]\tilde{u}_j(y)}^p\dm{y}\\
        &= \int_{\Omega}\abs{\sum_{j\in F}Q_\phi[\tilde{u}_j\circ\phi,f]\tilde{u}_j(\phi(x))}^p
            \abs{\det{D\phi(x)}}\dm{x}\\
        &\leq \norml{\det{D\phi(x)}}{\infty}{\Omega}
            \norml{\sum_{j\in F}Q_\phi[u^\phi_j,f] u^\phi_j}{p}{\Omega}^p,
    \end{align*}
    whence we arrive at \[
        \norml{\tilde{P}_Fg}{p}{\tilde{\Omega}} \leq \norml{\det{D\phi(x)}}{\infty}{\Omega}^{1/p}
            \norml{P^\phi_F f}{p}{\Omega}.
    \] The result now follows by applying the triangle inequality to
    $P^\phi_F f = P_Ff + P^\phi_F f - P_Ff$
    and noticing that
    \[
        \norml{f}{p}{\Omega}^p \leq \sup_{\tilde{\Omega}}\abs{\det{D\phi}^{-1}}\norml{g}{p}{\tilde{\Omega}}^p.\qedhere
    \]
\end{proof}

\section{Resolvent bounds for the square}\label{sec:resolvent}
In this section we will verify that \cref{eqn:resolventhyp} holds
when $\Omega$ is a square.

Recall that the Dirichlet eigenfunctions of $-\Delta$ on $[0,\pi]^2$ are $$u_{m,n}(x,y)=\sin(mx)\sin(ny)$$ for $(m,n)\in\N$,
with corresponding eigenvalues $\lambda_{m,n} = m^2+n^2$.

Thus, in $L^2$, we have \[
    (T-z)^{-1}u = \sum_{m,n}\frac{\hat{u}(m,n)}{\frac{1}{m^2+n^2}-z}u_{m,n},
\] where $\hat{u}(m,n)$ are the double-sine Fourier coefficients of $u\in L^2([0,\pi]^2)$,
provided that $z\in\C$ is not a sum of two squares.

Thus, the operator $(T-z)^{-1}$ can be viewed as a Torus Fourier multiplier with symbol \[
    \sigma_z(m,n) := \frac{m^2+n^2}{1-z(m^2+n^2)} \qfa m,n\in\N.
\]

Combining the Mihlin--H\"ormander Multiplier Theorem \citep[Theorem 6.2.7]{GrafakosCFA}
and a Transference Theorem \citep[Theorem 4.3.7]{GrafakosCFA}, we obtain the following result.
\begin{prop}\label{prop:multipliers}
    Suppose that $\chi \in L^\infty(\R^2)\cap C^2(\R^2)$ satisfies \[
        \abs{\partial^\alpha\chi(\xi)} \leq A\abs{\xi}^{-\abs{\alpha}}
    \] for all $\xi\in\R^2$ and all multiindices $\alpha$ with $\abs{\alpha}\leq 2$.
    Then, for all $1<p<\infty$, we have \[
        \norml{(\chi\hat{f})^{\lor}}{p}{\R^2} \leq C_p A\norml{f}{p}{\R^2}.
    \] Hence, for $1<p<\infty$, the operator \[
        f \mapsto \sum_{m,n}\chi(m,n)\hat{f}(m,n)u_{m,n}
    \] is bounded on $L^p([0,\pi]^2)$ with norm at most $C_pA$.
\end{prop}

We cannot apply \cref{prop:multipliers} directly to $\sigma_z$,
since $\sigma_z(\xi)$ is unbounded as $\xi \to \frac{1}{\sqrt{z}}$ in $\R^2$.
However, since we are only interested in $\sigma_z$ at integer lattice points $(m,n)$,
and by construction $1/z$ cannot be a sum of squares (since $z$ is not an eigenvalue of $-\Delta$),
it suffices to find another function that is an $L^p(\R^2)$ multiplier
and agrees with $\sigma_z$ on $\N^2$.

The solution is to select a smooth, radial cut-off function $\rho_z:[0,\infty)\to [0,1]$
which is zero on an annulus $\frac{1}{\sqrt{z}}-\epsilon < \abs{\xi} < \frac{1}{\sqrt{z}} + \epsilon$,
and $\rho_z\equiv 1$ outside the larger annulus
$\sqrt{\lambda_{\max F}}+\delta < \abs{\xi} <\sqrt{\lambda_{\max F+1}}-\delta$
(see \cref{fig:rho}).

\begin{figure}[ht]
    \begin{center}
        \includegraphics[scale=.3]{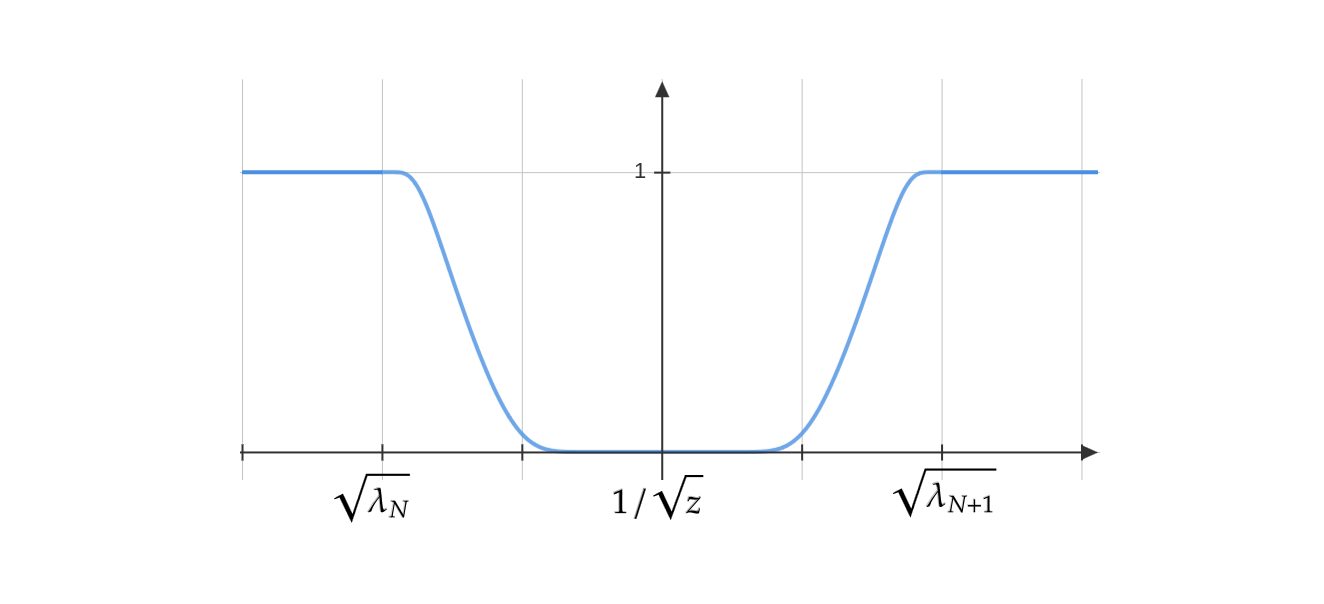}
        \caption{Graph of the smooth cut-off $\rho_z$ centred at $\abs{\xi}=\frac{1}{\sqrt{z}}.$}
        \label{fig:rho}
    \end{center}
\end{figure}

Finally, define $\chi_z(\xi) := \sigma_z(\xi)\rho_z(\abs{\xi})$.
This new symbol is smooth and bounded.
Furthermore, it agrees with $\sigma_z$ on lattice points, since $\rho_z=1$
at all such points.
Tedious calculations show that $\chi_z$ satisfies the assumptions of \cref{prop:multipliers},
and so it follows that $(T-z)^{-1}$ is bounded from $L^p$ to $L^p$.

\begin{remark}
    A note on the size of the constant $A$ in \cref{prop:multipliers}.
    If $F = \{1,\ldots, N\}$, then the ``worse'' value of $z$
    along the contour $\gamma$ in the proof of \cref{thm:deltaproj2}
    occurs when $\frac{1}{\lambda_{N+1}}< z < \frac{1}{\lambda_{N}}$,
    so the width of the ``window'' around $\frac{1}{\sqrt{z}}$ in \cref{fig:rho}
    will have to be comparable to $\sqrt{\lambda_{N+1}}-\sqrt{\lambda_N}$.
    In $\R^2$, as a consequence of Weyl's Law,
    $\lambda_N \approx N$ and $\lambda_{N+1}-\lambda_N \approx 1$ as $N\to\infty$.
    (Here $A \approx B$ means that $A/B$ remains bounded above and below by constants.)
    Therefore, the window is shrinking like $1/N$ as $N\to\infty$.
    This dependence on $N$ (\ie on $F$) manifests in the size of the derivatives of $\rho$.
    By our calculations, $A = O(N^5)$ when estimating the derivatives of $\chi_z$ up to second order.
\end{remark}

Finally, we note that we may also take $\Omega$ to be a rectangle,
by appropriate rescaling each direction.
\section{Conclusions and open problems}

This approach has two important limitations.
The first is the assumption that the resolvents $(-\Delta^{-1}-z)^{-1}$ be $L^p$
bounded.
As far as we know, the only general results for Banach spaces yield bounds
in terms of the distance of $z$ to the numerical range \citep[see][Theorem 6.11]{Carvalho2012}.
However, the numerical range contains the closed convex hull of the spectrum \citep[see]{Zenger68},
which is $[0,1/\lambda_1]$ for $-\Delta^{-1}$.
Thus, the distance from $z\in (1/\lambda_{N+1},1/\lambda_N)$ to the numerical range
is 0.

Another assumption we have to make is
that the projections $P_F$ map onto eigenspaces which do not split eigenvalues.
Eigenvalue splitting is not allowed because of Kato's projection formula,
as discussed in the remark immediately after \cref{thm:hilbertspaceproj}.

The following example motivates our interest in relaxing the notion of splitting.
\begin{example}\label{ex:cutoffs}
    Let $\Omega = [0,1]^2$, a unit square; then the eigenvalues of the Dirichlet Laplacian are
    $\pi^2(m^2+n^2)$ for $(m,n)$ running over $\N^2$. Fix any $N\in\N$.
    The following results are classical \citep[see][]{GrafakosCFA}:
    \begin{enumerate}
        \item If we take $F_N = \{(m,n) : 0< m, n \leq N\}$, then $P_{F_N}$ is the partial sum operator
        corresponding to ``square cutoff'' truncations of a double sine-series on the torus.
        For any value of $1<p<\infty$, these projections are uniformly bounded in $N$,
        which is equivalent to saying that the partial sums $P_{F_N}u$ converge to $u$ for any $u\in L^p$.
        \item If we take $B_N = \{(m,n) : m^2 + n^2 \leq N^2\}$, then $P_{B_N}$ is the partial sum operator
        corresponding to ``circular cutoff'' truncations of a double sine-series on the torus.
        These partial sums do not converge in any $L^p$ space other than $p=2$.
    \end{enumerate}

    It would be interesting to transfer the convergence results for (i) to other domains
    obtained by perturbing the square by some suitable $\phi$, as described here.
    However, by definition, the sets $B_N$ will not split eigenvalues,
    while the sets $F_N$ will often split eigenvalues.
    Take, for instance, the pairs $(5,5)$ and $(1,7)$,
    both corresponding to the eigenvalue $50\pi^2$.
    Then $(5,5)\in F_6$ but $(1,7)\notin F_6$.

    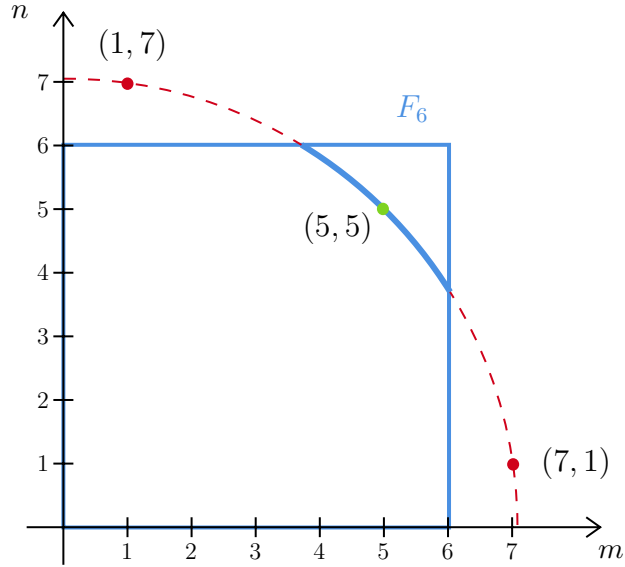
\begin{figure}[ht]
        \begin{center}
            \tikzset{every picture/.style={line width=0.75pt}} 

\begin{tikzpicture}[x=0.75pt,y=0.75pt,yscale=-1,xscale=1]

\draw  [draw opacity=0][fill={rgb, 255:red, 208; green, 2; blue, 27 }  ,fill opacity=1 ] (298.51,237) .. controls (298.51,235.27) and (299.86,233.88) .. (301.51,233.88) .. controls (303.17,233.88) and (304.51,235.27) .. (304.51,237) .. controls (304.51,238.73) and (303.17,240.13) .. (301.51,240.13) .. controls (299.86,240.13) and (298.51,238.73) .. (298.51,237) -- cycle ;
\draw  [color={rgb, 255:red, 74; green, 144; blue, 226 }  ,draw opacity=1 ][line width=1.5]  (77,76.25) -- (269.5,76.25) -- (269.5,268.75) -- (77,268.75) -- cycle ;
\draw  [draw opacity=0][fill={rgb, 255:red, 208; green, 2; blue, 27 }  ,fill opacity=1 ] (106,45.5) .. controls (106,43.77) and (107.34,42.38) .. (109,42.38) .. controls (110.66,42.38) and (112,43.77) .. (112,45.5) .. controls (112,47.23) and (110.66,48.63) .. (109,48.63) .. controls (107.34,48.63) and (106,47.23) .. (106,45.5) -- cycle ;
\draw  [draw opacity=0][dash pattern={on 4.5pt off 4.5pt}] (77.28,43) .. controls (77.47,43) and (77.65,43) .. (77.83,43) .. controls (202.24,43) and (303.16,143.6) .. (303.66,267.89) -- (77.83,268.83) -- cycle ; \draw  [color={rgb, 255:red, 208; green, 2; blue, 27 }  ,draw opacity=1 ][dash pattern={on 4.5pt off 4.5pt}] (77.28,43) .. controls (77.47,43) and (77.65,43) .. (77.83,43) .. controls (202.24,43) and (303.16,143.6) .. (303.66,267.89) ;
\draw  [draw opacity=0][line width=2.25]  (195.85,76.25) .. controls (226,94.77) and (251.46,120.16) .. (270.06,150.25) -- (77.83,268.83) -- cycle ; \draw  [color={rgb, 255:red, 74; green, 144; blue, 226 }  ,draw opacity=1 ][line width=2.25]  (195.85,76.25) .. controls (226,94.77) and (251.46,120.16) .. (270.06,150.25) ;
\draw  [draw opacity=0][fill={rgb, 255:red, 126; green, 211; blue, 33 }  ,fill opacity=1 ] (233.51,108.5) .. controls (233.51,106.77) and (234.86,105.38) .. (236.51,105.38) .. controls (238.17,105.38) and (239.51,106.77) .. (239.51,108.5) .. controls (239.51,110.23) and (238.17,111.63) .. (236.51,111.63) .. controls (234.86,111.63) and (233.51,110.23) .. (233.51,108.5) -- cycle ;
\draw  (58.67,268.6) -- (345,268.6)(77,9.75) -- (77,287.5) (338,263.6) -- (345,268.6) -- (338,273.6) (72,16.75) -- (77,9.75) -- (82,16.75) (109,263.6) -- (109,273.6)(141,263.6) -- (141,273.6)(173,263.6) -- (173,273.6)(205,263.6) -- (205,273.6)(237,263.6) -- (237,273.6)(269,263.6) -- (269,273.6)(301,263.6) -- (301,273.6)(72,236.6) -- (82,236.6)(72,204.6) -- (82,204.6)(72,172.6) -- (82,172.6)(72,140.6) -- (82,140.6)(72,108.6) -- (82,108.6)(72,76.6) -- (82,76.6)(72,44.6) -- (82,44.6) ;
\draw   (116,280.6) node[anchor=east, scale=0.75]{1} (148,280.6) node[anchor=east, scale=0.75]{2} (180,280.6) node[anchor=east, scale=0.75]{3} (212,280.6) node[anchor=east, scale=0.75]{4} (244,280.6) node[anchor=east, scale=0.75]{5} (276,280.6) node[anchor=east, scale=0.75]{6} (308,280.6) node[anchor=east, scale=0.75]{7} (74,236.6) node[anchor=east, scale=0.75]{1} (74,204.6) node[anchor=east, scale=0.75]{2} (74,172.6) node[anchor=east, scale=0.75]{3} (74,140.6) node[anchor=east, scale=0.75]{4} (74,108.6) node[anchor=east, scale=0.75]{5} (74,76.6) node[anchor=east, scale=0.75]{6} (74,44.6) node[anchor=east, scale=0.75]{7} ;

\draw (342.43,276.47) node [anchor=north west][inner sep=0.75pt]  [font=\small]  {$m$};
\draw (49.43,4.9) node [anchor=north west][inner sep=0.75pt]  [font=\small]  {$n$};
\draw (241.5,49.9) node [anchor=north west][inner sep=0.75pt]  [color={rgb, 255:red, 74; green, 144; blue, 226 }  ,opacity=1 ]  {$F_{6}$};
\draw (195,108.4) node [anchor=north west][inner sep=0.75pt]    {$( 5,5)$};
\draw (314,226.9) node [anchor=north west][inner sep=0.75pt]    {$( 7,1)$};
\draw (92.5,16.4) node [anchor=north west][inner sep=0.75pt]    {$( 1,7)$};

\end{tikzpicture}
            \caption{The ``box'' $F_6$ splits the eigenvalue $\lambda=50\pi^2$.
            The region $\Gamma_6$ is denoted by the solid blue arc.}
            \label{fig:badregion}
        \end{center}
    \end{figure}
    Denote this ``bad region'' by $\Gamma_N$, \ie\[
        \Gamma_N = \{(m,n)\in F_N : m^2+n^2 = k^2+l^2 \text{ for some } (k,l)\notin F_N\}.
    \]
    Then $F_N$ splits eigenvalues if, and only, if $\Gamma_N\neq\emptyset$
    (see \cref{fig:badregion}).
\end{example}
It is rather surprising that this ``small'' region $\Gamma_N$
lying between the circle and the square is responsible for such
dramatic differences in the convergence of the partial sums.
Indeed, Fefferman's proof shows that curvature is a key part of the failure
of $L^p$ convergence.
On the other hand, C\'ordoba's results show that varying the number of vertices,
we can still guarantee convergence in $L^p$ at least in the range $4/3<p<4$.

Our work in this paper has been motivated by the ``failure'' of the circle
to yield ``well-behaved'' (\ie convergent) cutoffs
\citep[see][concluding remarks]{RobinsonFefferman2022}.
Indeed, it would seem that circular cutoffs are the ``natural choice''
to truncate the partial sums since they correspond to choosing
all eigenfunctions whose eigenvalues are bounded by some value.
The square truncations, on the other hand, exploit the ``artificial''
labelling of these eigenvalues by pairs of integers.
Fefferman himself remarks on his surprise at his own theorem,
since it was conjectured that circular cutoffs would yield convergence
in the region $\frac{4}{3}<p<4$
\citep[see the introduction to][]{Fefferman1971}.

Thus, two open questions of especial interest remain:
\begin{enumerate}
    \item Can we eliminate the ``no-splitting'' requirement in order to apply our results
    to the square cutoffs?
    \item Can we improve the constant $C(F)$ in order to get uniform estimates on the operator norms?
\end{enumerate}

Regarding (ii), we observe that in some simple examples, there is \emph{no} constant $C(F)$,
and so we can immediately recover $L^p$ convergence on $\tilde{\Omega}$ from $L^p$ convergence on $\Omega$.
\begin{example}
    Let $\Omega = [0,\pi]^2$ with eigenfunctions $u_{m,n} = \sin(mx)\sin(ny)$,
    and $\phi(x,y) = (ax, by)$ mapping onto $\tilde{\Omega} = [0,\pi a]\times [0, \pi b]$.
    The eigenfunctions on the new domain are $\tilde{u}_{m,n} = \sin(mx/a)\sin(ny/b)$.
    A simple calculation then shows that \[
        \norml{P_F u - P_F^\phi u}{p}{\Omega} \leq \abs{1-ab}\norml{P_Fu}{p}{\Omega},
    \] and so we can recover $L^p$ boundedness of the \emph{square} projections.
\end{example}

\section*{Acknowledgments}
This work was conducted thanks to the EPSRC studentship 2443915 under the project EP/V520226/1.
For the purpose of open access, the authors have applied a Creative Commons Attribution (CC BY) license
to any Author Accepted Manuscript version arising from this submission.


\bibliography{bibliography}

\end{document}